\documentclass[11pt]{article}
\usepackage{url}
\usepackage{amsfonts}
\usepackage{amssymb}
\usepackage{amsmath}
\usepackage{amsthm}
\usepackage[lofdepth,lotdepth]{subfig}
\usepackage{setspace,amssymb,verbatim}
\usepackage{enumitem}
\usepackage[margin=1in]{geometry}
\usepackage{graphicx}
\usepackage{rotating}
\usepackage{psfrag}
\usepackage{afterpage}
\usepackage{multirow}
\usepackage[lofdepth,lotdepth]{subfig}
\usepackage{tikz}

\theoremstyle{plain}
\newtheorem{theorem}{Theorem}[section]
\newtheorem{proposition}[theorem]{Proposition}
\newtheorem*{theorem*}{Theorem}
\newtheorem{corollary}[theorem]{Corollary}

\newtheorem*{lemma*}{Lemma}
\newtheorem{lemma}[theorem]{Lemma}

\theoremstyle{remark}

\newtheorem{remark}[theorem]{Remark}

\newtheorem*{case*}{Case}

\theoremstyle{definition}

\usepackage{prettyref}
\newrefformat{sec}{Section \ref{#1}}
\newrefformat{chap}{Chapter \ref{#1}}
\newrefformat{prop}{Proposition \ref{#1}}
\newrefformat{prob}{Problem \ref{#1}}
\newrefformat{rem}{Remark \ref{#1}}
\newrefformat{cor}{Corollary \ref{#1}}
\newrefformat{cond}{Condition \ref{#1}}
\newrefformat{def}{Definition \ref{#1}}
\newrefformat{ex}{Example \ref{#1}}
\newrefformat{lem}{Lemma \ref{#1}}
\newrefformat{eq}{Equation \eqref{#1}}
\newrefformat{con}{condition (\ref{#1})}
\newrefformat{conj}{Conjecture \ref{#1}}
\newrefformat{fig}{Figure \ref{#1}}
\newrefformat{tab}{Table \ref{#1}}
\newrefformat{app}{Appendix \ref{#1}}

\newcommand{\stab}{\mathrm{stab}}

\newcommand{\Z}{\mathbb{Z}}
\newcommand{\Q}{\mathbb{Q}}

\newcommand{\C}{\mathbb{C}}

\newcommand{\F}{\mathbb{F}}
\newcommand{\K}{\mathbb{K}}

\newcommand{\nm}{\mathrm{Nm}}
\newcommand{\cP}{\mathcal{P}}
\newcommand{\cF}{\mathcal{F}}

\newcommand{\ind}{\mathrm{Ind}}
\newcommand{\res}{\mathrm{Res}}\newcommand{\Res}{\mathrm{Res}}
\newcommand{\irr}{\mathrm{Irr}}

\newcommand{\wh}[1]{\widehat{#1}}\newcommand{\wt}[1]{\widetilde{#1}}
\newcommand{\tr}{\mathrm{Tr}}

\newcommand{\gal}{\mathrm{Gal}}
\newcommand{\bg}[1]{\textbf{#1}}

\title{Fields of character values for finite special unitary groups}

\author{A. A. Schaeffer Fry\\ \small\textit{Department of Mathematical and Computer Sciences}\\\small\textit{Metropolitan State University of Denver}\\ \small\textit{Denver, CO 80217, USA} \\ \small\tt{aschaef6@msudenver.edu}
\and
C. Ryan Vinroot\\ \small\textit{Department of Mathematics}\\\small\textit{College of William and Mary}\\\small\textit{Williamsburg, VA 23187, USA}\\ \small\tt{vinroot@math.wm.edu}}
\date{}
\begin{document}
\maketitle
\begin{abstract}
Turull has described the fields of values for characters of $SL_n(q)$ in terms of the parametrization of the characters of $GL_n(q)$.  In this article, we extend these results to the case of $SU_n(q)$.
\end{abstract}

\section{Introduction}

It is a problem of general interest to understand the fields of values of the complex characters of finite groups, as these fields often reflect important or subtle properties of the group itself.  Turull \cite[Section 4]{turull01} computed the fields of character values of the finite special linear groups $SL_n(q)$ by using properties of degenerate Gelfand-Graev characters of $GL_n(q)$.  In this paper, we extend these methods to compute the fields of character values for the finite special unitary groups $SU_n(q)$.  In particular, we use properties of generalized Gelfand-Graev characters of $SU_n(q)$ and the full unitary group $GU_n(q)$ to get this information.  Further, we frame these methods so that we obtain many results for both $SL_n(q)$ and $SU_n(q)$ simultaneously.

Turull also computes the Schur indices of the characters of $SL_n(q)$.  This appears to be a much more difficult problem for $SU_n(q)$.  For example, it is helpful in the $SL_n(q)$ case that the Schur index for every character of $GL_n(q)$ is 1.  However, the Schur indices of the characters of $GU_n(q)$ are not all explicitly known, but are known to take values other than 1.

This paper is organized as follows.  In Section \ref{sec:Chars}, we establish the necessary results from character theory that are needed for the main arguments.  In Sections \ref{sec:LusztigInd} and \ref{sec:Param}, we give some tools from Deligne-Lusztig theory and the parameterization of the characters of $GL^{\epsilon}_n(q)$, respectively, and we use these to describe the characters of $SL^{\epsilon}_n(q)$ in Section \ref{sec:Restriction}.  We introduce generalized Gelfand-Graev characters in Section \ref{sec:GGGR}.  In Section \ref{sec:Initial}, we obtain some preliminary results on fields of character values which follow quickly from the material in Section \ref{sec:Chars}.  To deal with the harder cases, we need some explicit information on unipotent elements obtained in Section \ref{sec:Unipotent}, and we apply this information to generalized Gelfand-Graev characters in Section \ref{sec:MoreGGGR}.  Finally, in Section \ref{sec:Main} we prove our main results in Theorem \ref{thm:turullext1} and Corollary \ref{cor:turullext}, which give explicitly the fields of values of any character of $SU_n(q)$ and a description of the real-valued characters of $SU_n(q)$, as well as recover the corresponding results for $SL_n(q)$ originally found in \cite[Section 4]{turull01}.

\subsection*{Notation}

We will often use the notations found in \cite{turull01}, for clarity of analogous statements.  For example, the natural action of a Galois automorphism $\sigma\in\mathrm{Gal}(\bar{\Q}/\Q)$ on a character $\chi$ of a group will be denoted $\sigma\chi$.  Here for a group element $g$, the value of $\sigma\chi$ is given by $\sigma\chi(g)=\sigma(\chi(g))$.   We write $\Q(\chi)$ for the field obtained from $\Q$ by adjoining all  values of the character $\chi$.

For an integer $n$, we will write $n=n_2n_{2'}$ where $n_2$ is a $2$-power and $n_{2'}$ is odd. Further, for an element $x$ of a finite group $Y$, we write $x=x_2x_{2'}$ where $x_2$ has $2$-power order and $x_{2'}$ has odd order.  We denote by $|x|$ the order of the element $x$ (we also use this notation for cardinality and size of partitions, which will be clear from context).  We write $\irr(Y)$ for the set of all irreducible complex characters of the group $Y$.  Given two elements $g, x$ in $Y$, we write $g^x = x^{-1}gx$, and for $\chi \in \irr(Y)$, we define $\chi^x$ by $\chi^x(g) = \chi(xgx^{-1})$.  

For a subgroup $X\leq Y$, we write $\ind_X^Y(\varphi)$ for the character of $Y$ induced from a character $\varphi$ of $X$, and we write $\res^Y_X(\chi)$ for the character of $X$ restricted from a character $\chi$ of $Y$.  We will further use $\irr(Y|\varphi)$ and $\irr(X|\chi)$  to denote the set of irreducible constituents of $\ind_X^Y(\varphi)$ and $\res^Y_X(\chi)$, respectively. 

Throughout the article, let $q$ be a power of a prime $p$ and let $G=SL^\epsilon_n(q)$ and $\wt{G}=GL^\epsilon_n(q)$, where $\epsilon\in\{\pm1\}$.  Here when $\epsilon = 1$, we mean  $\wt{G}=GL_n(q)$ and $G=SL_n(q)$, and when $\epsilon=-1$, we mean $\wt{G}=GU_n(q)$ and $G=SU_n(q)$.   We also write $\bg{G}=SL_n(\bar{\F}_q)$ and $\wt{\bg{G}}=GL_n(\bar{\F}_q)$ for the corresponding algebraic groups, so that $\wt{G}=\wt{\bg{G}}^{F_\epsilon}$ and $G=\bg{G}^{F_\epsilon}$ for an appropriate Frobenius morphism $F_{\epsilon}$.

\section{Characters} \label{sec:Chars}

\subsection{Lusztig Induction} \label{sec:LusztigInd}

  For this section, we let $\mathbf{H}$ be any connected reductive group over $\bar{\F}_q$ with Frobenius map $F$, and write $H = \mathbf{H}^F$.  For any $F$-stable Levi subgroup $\mathbf{L}$ of $\mathbf{H}$, contained in a parabolic subgroup $\mathbf{P}$, we write $L = \mathbf{L}^F$ and denote by $R_L^H = R_{\mathbf{L} \subset \mathbf{P}}^{\mathbf{H}}$ the Lusztig (or twisted) induction functor.  When $\mathbf{P}$ may be chosen to be an $F$-stable parabolic, then $R_L^H$ becomes Harish-Chandra induction.  When $\mathbf{L} = \mathbf{T}$ is chosen to be a maximal torus and $\theta$ is a character of $T = \mathbf{T}^F$, then $R_T^H(\theta)$ is the corresponding Deligne-Lusztig (virtual) character.  We need the following basic result regarding actions on characters of finite reductive groups obtained through twisted induction.

\begin{lemma} \label{DLlemma}
Let $\mathbf{H}$ and $H = \mathbf{H}^F$ be as above.  Let $\mathbf{L}$ be an $F$-stable Levi subgroup of $\mathbf{H}$, and write $L = \mathbf{L}^F$.  Let $\chi$ be a character of $L$, $\sigma \in \gal(\bar{\Q}/\Q)$, and $\alpha$ a linear character of $H$ which is trivial on unipotent elements.  Then
$$ \sigma R_{L}^H(\chi) = R_L^H(\sigma \chi) \quad \text{ and } \quad \alpha R_L^H(\chi) = R_L^H(\alpha \chi).$$
In particular, when $\mathbf{L} = \mathbf{T}$ is a maximal torus and $\chi = \theta$ is a character of $T = \mathbf{T}^F$, then we have
$$ \sigma R_{T}^H(\theta) = R_T^H(\sigma \theta) \quad \text{ and } \quad \alpha R_T^H(\theta) = R_T^H(\alpha \theta).$$
\end{lemma}
\begin{proof} From \cite[Proposition 11.2]{dignemichel}, for any $g \in H$ we have
$$R_L^H(\chi)(g) = \frac{1}{|L|}\sum_{l \in L} \tr((g,l^{-1})|X) \chi(l),$$
where $\tr((g,l^{-1})|X)$ is the Lefschetz number corresponding to the $H \times L$-action on the $\ell$-adic cohomology $X$ of the relevant Deligne-Lusztig variety.  In particular, these numbers are rational integers (by \cite[Corollary 10.6]{dignemichel}, for example).  Thus,
$$\sigma R_L^H(\chi)(g) = \frac{1}{|L|}\sum_{l \in L} \tr((g,l^{-1})|X) \sigma\chi(l) = R_L^H(\sigma \chi).$$
Now let $g \in H$ have Jordan decomposition $g=su$, so that $s \in H$ is semisimple and $u \in H$ is unipotent, and we have $\alpha(g)=\alpha(s)$.  From \cite[Proposition 12.2]{dignemichel} we have 
\begin{equation} \label{DLGreen}
R_L^H(\chi)(g) = \frac{1}{|L||C_{\mathbf{H}}^{\circ}(s)^F|} \sum_{\{ h \in H \, \mid \, s \in h \mathbf{L} h^{-1} \}} |C_{h \mathbf{L}h^{-1}}^{\circ} (s)^F| \sum_{ v \in C_{h \mathbf{L}h^{-1}}^{\circ} (s)^F_{\mathrm{u}}} Q_{C_{h \mathbf{L}h^{-1}}^{\circ} (s)}^{C_{\mathbf{H}}^{\circ}(s)}(u, v^{-1})  \chi^h(sv),
\end{equation}
where $C_{\mathbf{H}}^{\circ}(s)$ denotes the connected component of the centralizer, and $Q_{C_{h\mathbf{L}h^{-1}}^{\circ}(s)}^{C_{\mathbf{H}}^{\circ}(s)}$ denotes the Green function.  Note that for any $h \in H$ from the first sum of \eqref{DLGreen}, and any unipotent $v$ from the second sum of \eqref{DLGreen}, we have 
$$\alpha(g) = \alpha(s) = \alpha(sv) = \alpha(h sv h^{-1}) =  \alpha^h(sv).$$
So, if we multiply \eqref{DLGreen} by $\alpha(g)$, we may pass this factor through the sums to obtain
$$ \alpha(g)  \chi^h(sv) = \alpha^h(sv)   \chi^h(sv)  = (\alpha \chi)^h(sv).$$
It follows that we have $\alpha(g) R_L^H(\chi)(g) = R_L^H(\alpha \chi)(g)$, as claimed.
\end{proof}

\subsection{Parametrization of Characters of $GL^\epsilon_n(q)$} \label{sec:Param}

We identify $GL_1(\bar{\F}_q)$ with $\bar{\F}_q^{\times}$, and so $F_{\epsilon}$ acts on $\bar{\F}_q^{\times}$ via $F_{\epsilon}(a)=a^{\epsilon q}$.  For any integer $k \geq 1$, we define $T_k$ to be the multiplicative subgroup of $\bar{\F}_q^{\times}$ fixed by $F_{\epsilon}^k$, that is
$$T_k = (\bar{\F}_q^{\times})^{F_{\epsilon}^k}.$$
We denote by $\wh{T}_k$ the multiplicative group of complex-valued linear characters of $T_k$.  Whenever $d|k$, we have the natural norm map $\nm_{k,d}=\nm$ from $T_k$ to $T_d$, and the transpose map $\wh{\nm}$ gives a norm map from $\wh{T}_d$ to $\wh{T}_k$, where $\wh{\nm}(\xi) = \xi \circ \nm$.  We consider the direct limit of the character groups $\wh{T}_k$ with respect to these norm maps, $\displaystyle \lim_{\longrightarrow} \wh{T_k}$,
on which $F_{\epsilon}$ acts through its natural action on the groups $T_k$.  Moreover, the fixed points of $\displaystyle \lim_{\longrightarrow} \wh{T_k}$ under $F_{\epsilon}^d$ can be identified with $\wh{T_d}$.  We let $\Theta$ denote the set of $F_{\epsilon}$-orbits of $\displaystyle \lim_{\longrightarrow} \wh{T_k}$.  The elements of $\Theta$ are sometimes called {\em simplices} (in \cite{Green, turull01} for example).  They are naturally dual objects to polynomials with roots given by an $F_{\epsilon}$-orbit of $\bar{\F}_q^{\times}$.

For any orbit $\phi \in \Theta$, let $|\phi|$ denote the size of the orbit.  Let $\cP$ denote the set of all partitions of non-negative integers, where we write $|\nu| = n$ if $\nu$ is a partition of $n$, and let $\cP_n$ denote the set of all partitions of $n$.  The irreducible characters of $\wt{G}=GL^{\epsilon}_n(q)$ are parameterized by partition-valued functions on $\Theta$.  Specifically, given a function $\lambda: \Theta \rightarrow \cP$, define $|\lambda|$ by
$$|\lambda| = \sum_{\phi \in \Theta} |\phi| |\lambda(\phi)|,$$
and define $\cF_n$ by
$$ \cF_n = \{ \lambda: \Theta \rightarrow \cP \, \mid \, |\lambda| = n\}.$$
Then $\cF_n$ gives a parametrization of the irreducible complex characters of $\wt{G}$.  Given $\lambda \in \cF_n$, we let $\wt{\chi}_{\lambda}$ denote the irreducible character corresponding to it.

We need several details regarding the structure of the character $\wt{\chi}_{\lambda}$.  In the case $\epsilon =1$, these facts all follow from the original work of Green \cite{Green}, and also appear from a slightly different point of view in the book of Macdonald \cite[Chapter IV]{MacBook}.  For the case $\epsilon=-1$, the facts we need appear in \cite{TV07}, which contains relevant results from \cite{dignemichelunitary, LuszSrin}.

First consider some $\lambda \in \cF_n$ such that $\lambda(\phi)$ is a nonempty partition for exactly one $\phi \in \Theta$, and write $\wt{\chi}_{\lambda}= \wt{\chi}_{\lambda(\phi)}$.  Suppose that $|\phi|=d$, so that $|\lambda(\phi)| = n/d$.  Then let $\omega^{\lambda(\phi)}$ be the irreducible character of the symmetric group $S_{n/d}$ parameterized by $\lambda(\phi) \in \cP_{n/d}$.  We fix this parametrization so that the partition $(1, 1, \ldots, 1)$ corresponds to the trivial character.  For any $\gamma =(\gamma_1, \gamma_2, \ldots, \gamma_{\ell}) \in \cP_{n/d}$, let $\omega^{\lambda(\phi)}(\gamma)$ denote the character $\omega^{\lambda(\phi)}$ evaluated at the conjugacy class parameterized by $\gamma$ (where $(1, 1, \ldots, 1)$ corresponds to the identity), and let $z_{\gamma}$ the size of the centralizer in $S_{n/d}$ of the class corresponding to $\gamma$.  Let $T_{\gamma}$ be the torus 
$$ T_{\gamma} = T_{d\gamma_1} \times T_{d \gamma_2} \times \cdots \times T_{d \gamma_{\ell}},$$
and let $\theta \in \phi$.  Then we have
\begin{equation} \label{DLLinComb}
\wt{\chi}_{\lambda(\phi)} = \pm \sum_{\gamma \in \mathcal{P}_{n/d}} \frac{ \omega^{\lambda(\phi)}(\gamma)}{z_{\gamma}} R_{T_{\gamma}}^{\wt{G}}(\theta),
\end{equation}
where the sign can be determined explicitly (see the Remark after \cite[Theorem 4.3]{TV07}, for example), but the sign will not have any impact for us.  Note that from \eqref{DLLinComb}, it follows from our parametrization of characters of the symmetric group and \cite[Proposition 12.13]{dignemichel} that the trivial character of $\wt{G}$ corresponds to $\lambda({\bf 1})=(1, 1, \ldots, 1)$.

For an arbitrary $\lambda \in \cF_n$, let $\phi_1, \phi_2, \ldots, \phi_r$ be precisely those elements in $\Theta$ such that $\lambda(\phi_i)$ is a nonempty partition, and let $d_i = |\phi_i|$.  Let $n_i = d_i |\lambda(\phi_i)|$, and define $L$ to be the Levi subgroup $L = GL_{n_1}^{\epsilon}(q) \times \cdots \times GL_{n_r}^{\epsilon}(q)$.  The character $\wt{\chi}_{\lambda}$ is then given by
\begin{equation} \label{Lusztigprod}
\wt{\chi}_{\lambda} = \pm R_{L}^{\wt{G}}\left(\wt{\chi}_{\lambda(\phi_1)} \times \cdots \times \wt{\chi}_{\lambda(\phi_r)}\right).
\end{equation}
The sign only appears in the $\epsilon = -1$ case, and again can be determined explicitly.  Note that \eqref{Lusztigprod} is Harish-Chandra induction in the case $\epsilon = 1$.

\subsection{Restriction to $SL^\epsilon_n(q)$ and Actions on the Parametrization} \label{sec:Restriction}

We now turn to the parametrization of the characters of $G = SL^{\epsilon}_n(q)$ in terms of the characters of $\wt{G}$ described in the previous section.  This is done for the case $\epsilon=1$ in \cite{karkargreen,Lehrer}, and we adapt the methods there to handle the more general case of $\epsilon = \pm 1$.

Consider any Galois automorphism $\sigma \in \gal(\bar{\Q}/\Q)$.  Through the natural action of $\sigma$ on the character values of any $\theta \in \wh{T}_d$, we have an action of $\sigma$ on the orbits $\phi \in \Theta$.  Given any $\lambda \in \cF_n$, we define $\sigma \lambda$ by 
$$ \sigma \lambda(\phi) = \lambda(\sigma \phi).$$
For $\alpha\in\wh{T}_1$, define $\alpha\wt{\chi}$ and $\alpha\theta$ by usual product of characters in $\irr(\wt{G})$ and $\wh{T}_d$, where we compose $\alpha$ with determinant and the norm maps, respectively.  Then $\alpha$ acts on the orbits $\phi \in \Theta$ as well, and we get an action of $\alpha$ on $\mathcal{F}_n$ by defining $\alpha\lambda$ as
$$\alpha\lambda(\phi)=\lambda(\alpha\phi).$$
We will need the following statements regarding these actions on the characters of $\wt{G}$.

\begin{lemma}\label{lem:actions}
Let $\lambda\in\mathcal{F}_n$.  For any $\sigma \in \gal(\bar{\Q}/\Q)$ and any $\alpha \in \wh{T}_1$, we have
$$ \sigma \wt{\chi}_{\lambda} = \wt{\chi}_{\sigma \lambda} \quad \text{ and } \quad \alpha \wt{\chi}_{\lambda} = \wt{\chi}_{\alpha \lambda}.$$
\end{lemma}
\begin{proof}  We proceed in a manner similar to the proof of \cite[Proposition of Section 3]{karkargreen}, which proves this statement in the $\epsilon=1$ case with the $\wh{T}_1$ action.  We begin by considering $\lambda \in \mathcal{F}_n$ such that $\lambda(\phi)$ is a nonempty partition for precisely one $\phi \in \Theta$, and so $\wt{\chi}_{\lambda}=\wt{\chi}_{\lambda(\phi)}$ is given by \eqref{DLLinComb}.  Given $\sigma \in \gal(\bar{\Q}/\Q)$, since each $\omega^{\lambda(\phi)}(\gamma)$ and each $z_{\gamma}$ is a rational integer, we have
$$\sigma \wt{\chi}_{\lambda(\phi)} = \pm \sum_{\gamma \in \mathcal{P}_{n/d}} \frac{ \omega^{\lambda(\phi)}(\gamma)}{z_{\gamma}} \sigma R_{T_{\gamma}}^{\wt{G}}(\theta) = \pm \sum_{\gamma \in \mathcal{P}_{n/d}} \frac{ \omega^{\lambda(\phi)}(\gamma)}{z_{\gamma}} R_{T_{\gamma}}^{\wt{G}}(\sigma\theta),$$
by Lemma \ref{DLlemma}.  Since $\sigma \theta \in \sigma \phi$, then we have $\sigma \wt{\chi}_{\lambda(\phi)} = \wt{\chi}_{\sigma \lambda(\phi)}$.  Similarly, if $\alpha \in \wh{T}_1$, we have by Lemma \ref{DLlemma}
$$\alpha \wt{\chi}_{\lambda(\phi)} = \pm \sum_{\gamma \in \mathcal{P}_{n/d}} \frac{ \omega^{\lambda(\phi)}(\gamma)}{z_{\gamma}} \alpha R_{T_{\gamma}}^{\wt{G}}(\theta) = \pm \sum_{\gamma \in \mathcal{P}_{n/d}} \frac{ \omega^{\lambda(\phi)}(\gamma)}{z_{\gamma}} R_{T_{\gamma}}^{\wt{G}}(\alpha\theta),$$
and since $\alpha \theta \in \alpha \phi$, we have $\alpha \wt{\chi}_{\lambda(\phi)} = \wt{\chi}_{\alpha \lambda(\phi)}$.

Now consider an arbitrary $\lambda \in \mathcal{F}_n$, with $\wt{\chi}_{\lambda}$ given by \eqref{Lusztigprod}.  By applying Lemma \ref{DLlemma}, along with the first case just proved, we have
\begin{align*}
\sigma \wt{\chi}_{\lambda} & = \pm R_L^{\wt{G}}\left(\sigma(\wt{\chi}_{\lambda(\phi_1)} \times \cdots \times \wt{\chi}_{\lambda(\phi_r)})\right) \\
& = \pm R_L^{\wt{G}}\left(\sigma\wt{\chi}_{\lambda(\phi_1)} \times \cdots \times \sigma\wt{\chi}_{\lambda(\phi_r)})\right) \\
& = \pm R_L^{\wt{G}}\left(\wt{\chi}_{\sigma\lambda(\phi_1)} \times \cdots \times \wt{\chi}_{\sigma\lambda(\phi_r)}\right) \\
& = \wt{\chi}_{\sigma \lambda}.
\end{align*}
Similarly, if we replace $\sigma$ with $\alpha \in \wh{T}_1$, we have $\alpha \wt{\chi}_{\lambda} = \wt{\chi}_{\alpha \lambda}$ as claimed.
\end{proof}

Note that we may identify $\wt{G}/G$ with $T_1$, and directly from Clifford theory we know every character $\chi$ of $G$ appears in some multiplicity-free restriction of a character $\wt{\chi}_{\lambda}$ of $\wt{G}$.  The restrictions of two different irreducible characters of $\wt{G}$ are either equal, or have no irreducible constituents of $\irr(G)$ in common.  With this, the next result is all that is needed to parameterize $\irr(G)$.  

\begin{lemma}\label{lem:irrparam}
Let $\lambda, \mu \in \cF_n$ with $\wt{\chi}_{\lambda}, \wt{\chi}_{\mu} \in \irr(\wt{G})$ the corresponding characters.  Then
$$\Res^{\wt{G}}_G (\wt{\chi}_{\lambda}) = \Res^{\wt{G}}_G (\wt{\chi}_{\mu})$$
if and only if there exists some $\alpha \in \wh{T}_1$ such that $\lambda = \alpha \mu$.
\end{lemma}
\begin{proof}This follows directly from \cite[Theorem 1(i)]{karkargreen} and Lemma \ref{lem:actions}.
\end{proof}

Consider any irreducible character $\chi$ of $G$, so $\chi$ is a constituent of $\Res^{\wt{G}}_G(\wt{\chi}_{\lambda})$ for some $\lambda \in \cF_n$.  That is, $\chi\in\irr(G|\wt{\chi}_\lambda)$.  The other constituents of this restriction are $\wt{G}$-conjugates of $\chi$.  Note that the field of values of $\chi$ is invariant under conjugation by $\wt{G}$, and so in studying this field of character values it is not important which constituent we choose.  

Given $\lambda \in \cF_n$, we define the group $\mathcal{I}(\lambda)$ as
$$\mathcal{I}(\lambda)=\bigcap\{\ker\alpha \, \mid \, \alpha\in\wh{T}_1 \hbox{ such that } \alpha\lambda=\lambda\}.$$
We collect some basic properties of $\mathcal{I}(\lambda)$ in the following.

\begin{proposition}\label{prop:indstabdivides}
Let $\chi\in\irr(G|\wt{\chi}_\lambda)$.  Then: 
\begin{itemize}
\item The stabilizer in $\wt{G}$ of $\chi$ is the set of elements with determinant in $\mathcal{I}(\lambda)$.
\item The stabilizer of $\lambda$ in $\wh{T}_1$ is the set of elements whose kernel contains $\mathcal{I}(\lambda)$.  
\item The index $[T_1: \mathcal{I}(\lambda)]$ divides $\gcd(q-\epsilon,n)$.  
\end{itemize}
\end{proposition}

\begin{proof}
The proof is exactly as in \cite[Propositions 4.2 and 4.3 and Corollary 4.4]{turull01}, using Clifford theory and \prettyref{lem:actions}. 
\end{proof}

\subsection{Remarks on Generalized Gelfand-Graev Characters}\label{sec:GGGR}

We recall here some subgroups described in \cite[Section 2]{Geck04} used in the construction of the characters of generalized Gelfand-Graev representations (GGGRs). We introduce only the essentials for our purposes, and refer the reader to \cite{Geck04, kawanaka1985, Taylor16}, for example, for more details.

First, let $\bg{T}\leq \bg{B}=\bg{T}\bg{U}$ be an $F_\epsilon$-stable maximal torus and Borel subgroup, respectively, of $\bg{G}$, with unipotent radical $\bg{U}$.  Let  $\Phi$ be the root system of $\bf{G}$ with respect to $\bg{T}$ and $\Phi^+\subset \Phi$ the set of positive roots determine by $\bg{B}$. 
 
To each unipotent class $\mathcal{C}$ in $\bg{G}$ (or, equivalently, in $\wt{\bg{G}}$), there is associated a weighted Dynkin diagram $d\colon \Phi\rightarrow\Z $ and $F_\epsilon$-stable groups
\[\bg{U}_{d,i}:=\langle X_{\alpha}|\alpha\in\Phi^+, d(\alpha)\geq i\rangle \leq \bg{U},\]
where $X_\alpha$ denotes the root subgroup corresponding to $\alpha$.  In particular, $\bg{P}_d:=N_{\wt{\bg{G}}}(\bg{U}_{d,1})$ is an $F_\epsilon$-stable parabolic subgroup of $\wt{\bg{G}}$ and $\bg{U}_{d,i}\lhd \bg{P}_d$ for each $i=1,2,3,\ldots$.  We will further write $U_{d,i}:=\bg{U}_{d,i}^{F_\epsilon}$ and $P_d:=\bg{P}_d^{F_\epsilon}$.
 
Given $u\in\mathcal{C}\cap U_{d,2}$, the characters of GGGRs (which we will also refer to as GGGRs)  of $\wt{G}$, respectively $G$, are constructed by inducing certain linear characters $\varphi_u\colon U_{d,2}\rightarrow\C^\times$ to $\wt{G}$, resp. ${G}$.  In particular, the values of $\varphi_u$ are all $p$th roots of unity.  Strictly speaking, the GGGRs are actually rational multiples of the induced character:
\[\wt{\Gamma}_u=[U_{d,1}:U_{d,2}]^{-1/2}\ind_{U_{d,2}}^{\wt{G}}(\varphi_u)\quad\hbox{ and }\quad \Gamma_u=[U_{d,1}:U_{d,2}]^{-1/2}\ind_{U_{d,2}}^{G}(\varphi_u).\]
 
 The following is \cite[Proposition 10.11]{SFTaylorTypeA}, which is a consequence of  \cite[Theorem 1.8, Lemma 2.6, and Theorem 10.10]{tiepzalesski04}.

\begin{proposition}\label{prop:GGGRvalues}
Let $\Gamma_u$ be a GGGR of $G$. Then the following hold.
\begin{enumerate}
\item If $q$ is a square, $n$ is odd, or $n/(n,q-\epsilon)$ is even, then the values of $\Gamma_u$ are integers.
\item Otherwise, the values of $\Gamma_u$ lie in $\Q(\sqrt{\eta p})$, where $\eta\in\{\pm1\}$ is such that $p\equiv\eta\mod 4$. 
\end{enumerate}
 \end{proposition}
 
\section{Initial Results on Fields of Values} \label{sec:Initial}

Keep the notation from above, so that $G=SL^\epsilon_n(q)$, $\wt{G}=GL^\epsilon_n(q)$, and the characters of $\wt{G}$ are denoted by $\wt{\chi}_\lambda$ for $\lambda\in\mathcal{F}_n$.  
For $\lambda\in\mathcal{F}_n$, let $\Q(\lambda)$ denote the field obtained from $\Q$ by adjoining the values of the characters in the orbits $\phi\in\Theta$ such that $\lambda(\phi)$ is nonempty. 

We define $\mathrm{Galg}(\lambda)$ and $\mathrm{Galr}(\lambda)$ as in \cite{turull01}.  That is, $\mathrm{Galg}(\lambda)$ is the stabilizer of $\lambda$ in $\gal(\Q(\lambda)/\Q)$ and 
\[\mathrm{Galr}(\lambda)=\{\sigma\in\gal(\Q(\lambda)/\Q)\, \mid \, \sigma\lambda=\alpha\lambda \hbox{ for some $\alpha\in\wh{T}_1$}\}.\]

\begin{theorem}
Let $\lambda\in\mathcal{F}_n$.  Then 
$\Q(\wt{\chi}_\lambda)=\Q(\lambda)^{\mathrm{Galg}(\lambda)}$ and $ \Q(\res^{\wt{G}}_G(\wt{\chi}_\lambda))=\Q(\lambda)^{\mathrm{Galr}(\lambda)}.$  That is, the field of values for $\wt{\chi}_\lambda$ and its restriction to $G$ are the fixed fields of $\mathrm{Galg}(\lambda)$ and $\mathrm{Galr}(\lambda)$, respectively.
\end{theorem}

\begin{proof}
Given \prettyref{lem:actions}, the proof is  exactly the same as that of \cite[Propositions 2.8, 3.4]{turull01}.
\end{proof}

 Note that since the members of $\phi\in\Theta$ are characters of $T_d$ for some $d$, it follows that $\Q(\lambda)=\Q(\zeta_m)$ is the field obtained from $\Q$ by adjoining some primitive $m$th root of unity $\zeta_m$, where $\gcd(m,p)=1$.
 
\begin{remark}\label{rem:sigmainv}
We further remark that, as in the proof of \cite[Proposition 6.2]{turull01}, the Galois automorphism $\sigma_{-1}\colon\Q(\zeta_m)\rightarrow\Q(\zeta_m)$ satisfying $\sigma_{-1}(\zeta_m)=\zeta_m^{-1}$ induces complex conjugation on $\Q(\lambda)$. Hence $\wt{\chi}_\lambda$, respectively $\res^{\wt{G}}_G(\wt{\chi}_\lambda)$, is real-valued if and only if $\sigma_{-1}\in \mathrm{Galg}(\lambda)$, respectively $\sigma_{-1}\in\mathrm{Galr}(\lambda)$.
\end{remark}

For the remainder of the article,  we let $\F_\lambda$ denote the field of values of  $\res^{\wt{G}}_G(\wt{\chi}_\lambda)$.  That is, $\F_\lambda$ is the fixed field of $\mathrm{Galr}(\lambda)$.  Since $\F_\lambda\subseteq\Q(\lambda)=\Q(\zeta_m)$  and $\gcd(m,p)=1$, we have $\F_\lambda\cap \Q(\zeta_p)=\Q$ for any primitive $p$th root of unity $\zeta_p$.

 \begin{proposition}\label{prop:initialcase}
 Let $\chi\in\irr(G|\wt{\chi}_\lambda)$.  Keep the notation above.  Then 
 \begin{enumerate}
 \item If $q$ is square, $n$ is odd, or $n/(n,q-\epsilon)$ is even, then $\Q(\chi)=\F_\lambda$.
 \item Otherwise, $\F_\lambda\subseteq\Q(\chi)\subseteq \F_\lambda(\sqrt{\eta p})$, where $\eta\in\{\pm1\}$ is such that $p\equiv\eta\mod 4$.
 \end{enumerate}
 In particular, $\chi$ is real-valued if and only if $\res^{\wt{G}}_G(\wt{\chi}_\lambda)$ is, except possibly when $q\equiv 3\mod4$ and $2\leq n_2\leq (q-\epsilon)_2$.
 \end{proposition}
 
 \begin{proof}

 Write $\F:=\F_\lambda$ and $\wt{\chi}:=\wt{\chi}_\lambda$.  First, we remark that certainly $\F\subseteq \Q(\chi)$, by its definition, since $\Res_G^{\widetilde{G}}(\widetilde{\chi})$ is the sum of $\wt{G}$-conjugates of $\chi$.  
  
Let $\wt{\Gamma}$ be a GGGR of $\wt{G}$ such that such that $\langle \widetilde{\Gamma}, \widetilde{\chi}\rangle_{\widetilde{G}} = 1$, which exists by a well-known result of Kawanaka (see \cite[3.2.18]{kawanaka1985} or \cite[15.7]{Taylor16}). Further, there exists a GGGR, $\Gamma$, of $G$ such that $\widetilde{\Gamma} = \ind_G^{\widetilde{G}}(\Gamma)$. Then Frobenius reciprocity yields that there is a unique irreducible constituent $\chi_0\in\irr(G|\wt{\chi})$ satisfying $\langle \Gamma, \chi_0 \rangle_G = 1$.  Without loss, we may assume $\chi$ is this $\chi_0$, as the field of values is invariant under $\wt{G}$-conjugation.

Write $\K=\F(\sqrt{\eta p})$.  Let $\sigma\in\gal(\overline{\Q}/\F)$ in case (1), and let $\sigma\in\gal(\overline{\Q}/\K)$ in case (2).  Then by \prettyref{prop:GGGRvalues}, $\sigma\chi$ is also a constituent of $\Gamma = \sigma\Gamma$ occurring with multiplicity 1. However, as $\Res_G^{\wt{G}}(\wt{\chi})$ is invariant under $\sigma$, we have $\sigma\chi$ is also a constituent of the restriction $\Res_G^{\wt{G}}(\wt{\chi})$.  Hence we see that $\sigma\chi=\chi$, by uniqueness, and hence $\Q(\chi)\subseteq \F$ in case (1) and $\Q(\chi)\subseteq \K$ in case (2).  
\end{proof}
 
In our main results below, characters of $T_1$ of $2$-power order will play an important role.  In particular, we denote by $\mathrm{sgn}$  the unique member of $\wh{T}_1$ of order $2$.
 
 \begin{lemma}\label{lem:orbiteven}
Let $\chi\in\irr(G|\wt{\chi}_\lambda)$ and write $I:=\stab_{\wt{G}}(\chi)$.  Then $[\wt{G}:I]$ is even if and only if $\mathrm{sgn}\lambda=\lambda$.
\end{lemma}
\begin{proof}
Note that $2$ divides $[\wt{G}:I]$ if and only if $[I:G]_2\leq \frac{1}{2}(q-\epsilon)_2$, if and only if $\mathcal{I}(\lambda)$ is contained in the unique subgroup of $\wt{G}/G$ of order $\frac{1}{2}(q-\epsilon)$.  But notice that this is exactly the kernel of $\mathrm{sgn}$ as an element of $\wh{T_1}$.
\end{proof}

 \begin{lemma}\label{lem:sgn}
Let $\chi\in\irr(G|\wt{\chi}_\lambda)$. If $\mathrm{sgn}\lambda\neq\lambda$, then $\F_\lambda=\Q(\chi)$.
 \end{lemma}
 \begin{proof}
Write $\F:=\F_\lambda$ and recall that $\F\subseteq\Q(\chi)\subseteq \F(\sqrt{\eta p})$.  Let $\K=\F(\sqrt{\eta p})$, so that $\K$ is a quadratic extension of $\F$.  Let $\tau$ be the generator of $\gal(\K/\F)$ and write $I$ for the stabilizer of $\chi$ under $\wt{G}$.  Then note that $\tau^2$ necessarily fixes $\chi$, and by definition $\tau$ fixes $\Res^{\wt{G}}_G(\wt{\chi})$, which by Clifford theory is the sum of the $[\wt{G}:I]$ conjugates of $\chi$ under the action of $\wt{G}$.  
 
We prove the contrapositive. Suppose $\F\neq\Q(\chi)$, so that $\tau$ does not fix $\chi$.  Then since the field of values is invariant under $\wt{G}$-conjugation, it follows that the orbit of $\chi$ under $\wt{G}$ can be partitioned into pairs conjugate to $\{\chi, \tau\chi\}$.  Hence the size of the orbit, $[\wt{G}:I]$, must be even, so $\mathrm{sgn}\lambda=\lambda$ by \prettyref{lem:orbiteven}.  
 \end{proof}
 
 \section{Unipotent Elements}  \label{sec:Unipotent}

To deal with the remaining cases (in particular, when
$q\equiv \eta\pmod 4$ is nonsquare, $\eta\in\{\pm1\}$, $\epsilon=-1$, and $2\leq n_2\leq (q+1)_2$), we will continue to employ the use of GGGRs.  For this, we will need to analyze certain aspects of conjugacy of unipotent elements.  Here the authors' observations in \cite{SFVinroot} on this subject will be useful.

In particular, if a unipotent element of $\wt{G}$ has $m_k$ Jordan blocks of size $k$ (that is, $m_k$ elementary divisors of the form $(t-1)^k$), then we may find a conjugate in $\wt{G}$ of the form $\bigoplus_k \wt{J}_k^{m_k}$, where the sum is over only those $k$ such that $m_k \neq 0$ and each $\wt{J}_k\in GL^\epsilon_k(q)$.  The following lemma, which is \cite[Lemma 3.2]{SFVinroot}, will be useful throughout the section.

\begin{lemma}\label{lem:SFV3.2}
Let $u$ be a unipotent element in $\wt{G}$ with $m_k$ Jordan blocks of size $k$ for each $1\leq k\leq n$.  For each $k$ such that $m_k\neq 0$, let $\delta_k\in T_1$ be arbitrary.  Then there exists some $g \in C_{\wt{G}}(u)$ such that $\mathrm{det}(g) = \prod_k \delta_k^k$.
\end{lemma}

 Let $\zeta_p$ be a primitive $p$th root of unity in $\C$.  In what follows, we let $b$ be a fixed integer such that $\mathrm{Gal}(\Q(\zeta_p)/\Q)$ is generated by the map $\tau\colon \zeta_p\mapsto \zeta_p^b$.  Note that $(b,p)=1$ and $b$ has multiplicative order $p-1$ modulo $p$.  Further, note that $\tau$ also induces the map $\sqrt{\eta p}\mapsto-\sqrt{\eta p}$ generating $\mathrm{Gal}(\Q(\sqrt{\eta p})/\Q)$.  Let $\bar{b}$ denote the image of $b$ under a fixed isomorphism $(\Z/p\Z)^\times \rightarrow \F_p^\times$, so that $\bar{b}$ generates $\F_{p}^\times$.
 
 Note that by \cite[Theorem 1.9]{tiepzalesski04}, every unipotent element $u$ of $GU_n(q)$ is conjugate to $u^b$ in $C_{\widetilde{G}}(s)$, where $s$ is a semisimple element in $C_{\wt{G}}(u)$.   We are interested in making precise statements about such a conjugating element. 
 
 To begin, let $u$ be a regular unipotent element of $GU_n(q)$, identified as in \cite[Lemma 5.1]{SFVinroot}.  Arguing as there, we see that an element conjugating $u$ to $u^b$ must have diagonal \[(\bar{b}^{n-1}\beta, \bar{b}^{n-2}\beta,...,\bar{b}\beta,\beta),\] where $\beta\in\F_{q^2}^\times$ and $\bar{b}^{n-1}\beta^{q+1}=1$.  
Note that the determinant of such an element is $\bar{b}^{\binom{n-1}{2}}\beta^n$ and that the condition that $\bar{b}^{n-1}\beta^{q+1}=1$ yields that $\beta^{q+1}$ is a $(p-1)$-root of unity.  
 
 \begin{lemma}\label{lem:regunip}\label{lem:beta2}
Let  $q\equiv \eta\mod 4$ be nonsquare with $\eta\in\{\pm1\}$ and let
  $u$ be a regular unipotent element of $GU_n(q)$. Keep the notation above.   Then
  \begin{enumerate}[label=(\alph*)]
  \item If $n$ is even, then $\beta_2$ is a primitive $(q^2-1)_2$-root of unity in $\F_{q^2}^\times$.
  \item There is an element $x$ in $GU_n(q)$ such that $u^x=u^b$ and $|\det(x)|$ is a $2$-power.  
  \item If $n\not\equiv 0 \pmod 4$, there is an element $x$ in $GU_n(q)$ such that $u^x=u^b$ and $|\det(x)| = (q+1)_2$.
 \end{enumerate}
 \end{lemma}
 \begin{proof}
 
For part (a) note that  $|\bar{b}^{n-1}|$ has the same $2$-part as $|\bar{b}|$ since $n$ is even, so $|\beta_2^{q+1}|=(p-1)_2=(q-1)_2$ since $q$ is nonsquare.  Hence the multiplicative order of $\beta_2$ is $2(q-\eta)_2=(q^2-1)_2$.
 
To prove the rest, we begin by showing that if $n\not\equiv 0\pmod4$, then we may find an $x$ in $GU_n(q)$ such that $u^x=u^b$ and $|\det(x)|_2=(q+1)_2$. If $n\equiv 2\pmod 4$, then  $\beta_2$ is a primitive $(q^2-1)_2$-root of unity by (a), and hence $\beta_2^n$ is a $(q-\eta)_2$ root of unity.  If $\eta=-1$, then since $|\bar{b}|_2=2,$ we see $\bar{b}^{n-1\choose 2}$ has odd order.  Then any $x\in GU_n(q)$ satisfying $u^x=u^b$ must satisfy $|\det(x)|_2 = (q+1)_2$ in this case. If $\eta=1$, note that $(q+1)_2=2$, so we must just show that $\det(x)$ has even order. Here $|\bar{b}|_2=(q-1)_2,$ and $\bar{b_2}^{\binom{n-1}{2}}$ has order strictly smaller than $\bar{b}_2$.  Hence $\bar{b}^{\binom{n-1}{2}}\cdot\beta^n$ has even order, so $|\det(x)|_2 = (q+1)_2$ again in this case.
 
Now assume $n$ is odd and let $\wt{x}$ be an element in $GU_n(q)$ satisfying $u^{\wt{x}}=u^b$.  Then certainly $\det(\wt{x})\in {T}_1$, so we may use \prettyref{lem:SFV3.2} to replace $\wt{x}$ with some $x\in GU_n(q)$ satisfying $\det(x)=\det(\wt{x})\cdot \delta^n$ for any $\delta\in {T}_1$.  In particular, note that $|\delta^n|=|\delta|$ for any $(q+1)_2$-root of unity $\delta$, since $n$ is odd.  Then we may choose $\delta$ so that $\det(\wt{x})_2\delta^n$ is a primitive $(q+1)_2$-root of unity, yielding  $|\det(x)|_2 = (q+1)_2$.

It remains to show that in all cases, $x$ can be chosen such that $|\det(x)|_{2'}=1$.  Since $\beta^{q+1}$ is a $(p-1)$-root of unity, we may decompose the determinant of $x$ into $\beta_2^n\cdot \beta_{(q+1)_{2'}}^n\cdot y$, where $y$ is a $(p-1)$-root of unity.  However, we also know that the determinant is a $(q+1)$-root and an odd prime cannot divide both $p-1$ and $q+1$.  Hence $y$ must be a $2$-power root of unity, and we may replace $x$ with an element of determinant $\beta_2^n\cdot y$, using \prettyref{lem:SFV3.2}.
 \end{proof}
 
 We remark that arguing similarly, we see that in fact there is no $x$ satisfying the conclusion of \prettyref{lem:regunip}(c) if $u$ is a regular unipotent element when $4$ divides $n$.  We can, however, generalize to the following statement about more general unipotent elements when $4\not|n$.
 
 \begin{corollary}\label{cor:unipconj}
Let $q\equiv \eta\mod 4$ be nonsquare with $\eta\in\{\pm1\}$.  If $u$ is a unipotent element of $GU_n(q)$ satisfying at least one of the following:
\begin{enumerate} 
\item $u$ has an odd number of elementary divisors of the form $(t-1)^k$ with $k\equiv 2\mod 4$;
\item $u$ has an elementary divisor of the form $(t-1)^k$ with $k$ odd,
\end{enumerate}
then $u$ is conjugate to $u^b$ by an element $x$ satisfying $|\det(x)| = (q+1)_2$.  

In particular, if $n$ is not divisible by $4$, any unipotent element is conjugate to $u^b$ by an element $x$ satisfying $|\det(x)| = (q+1)_2$.
 \end{corollary}
 \begin{proof}
 Indeed, viewing $u$ as $\bigoplus_k \wt{J}_k^{m_k}$ as in \cite[Section 3.2]{SFVinroot}, we may find elements $x_k$ for each $1\leq k\leq n$ as in \prettyref{lem:regunip} conjugating each $\tilde{J}_k$ to $\tilde{J}_k^b$. In case (1), we see that the product $\bigoplus_k{x_k}^{m_k}$ will satisfy the statement, after possibly again using \prettyref{lem:SFV3.2} to replace $x_k$ for any odd $k$ with an element satisfying $|\det(x_k)|=1$.
 
 If (2) holds, but (1) does not hold, $y=\bigoplus_{2|k} {x_k}^{m_k}$ will satisfy $|\det(y)|=|\det(y)|_2 < (q+1)_2$.  We may use \prettyref{lem:regunip} to obtain $x_k$ for some $k$ odd  such that $|\det(x_k)| = (q+1)_2$, and replace the remaining $x_k$ for odd $k$ with an element satisfying $|\det(x_k)|=1$.  The resulting $\bigoplus_k{x_k}^{m_k}$ will satisfy the statement.
 
 The last statement follows, since if $n$ is odd, we must be in case (2), and if $n\equiv2\pmod4$, we must be in case (1) or (2).
\end{proof}
 
 \begin{remark}\label{rem:condsnmod4}
 We remark that at least one of conditions 1 and 2 of \prettyref{cor:unipconj} must occur if $n\equiv 2\pmod 4$, and that condition 1 implies condition 2 if $n\equiv 0\mod 4$. Further, when $\eta=1$, the condition $n_2\leq (q+1)_2$ induced from \prettyref{prop:initialcase} yields that $n\equiv2\pmod 4$.  
 
 \end{remark}
 
 We now address the case that $4$ divides $n$, $q\equiv 3\pmod 4$, and that neither of the conditions in \prettyref{cor:unipconj} occur.
 
 \begin{lemma}\label{lem:unipconj0mod4}
Let $q\equiv 3\mod 4$ and let $n\equiv 0\pmod 4$ such that $n_2\leq (q+1)_2$.  Let $u$ be a unipotent element of $GU_n(q)$ with no elementary divisors $(t-1)^k$ with $k$ odd.  Then $u$ is conjugate to $u^b$ by an element $x$ satisfying $|\det(x)| = \frac{(q^2-1)_2}{n_2}$.  \end{lemma}
 \begin{proof}
 
 As in the proof of \prettyref{cor:unipconj}, let $\wt{x}=\bigoplus_k{x_k}^{m_k}$, where for each $k$ such that $m_k\neq 0$, $x_k$ is an element of $GU_k(q)$ conjugating $\tilde{J}_k$ to $\tilde{J}_k^b$ as in \prettyref{lem:regunip}.  Now, each $x_k$ has determinant $\pm{(\beta_k)_2}^k $, where ${(\beta_k)_2}$ is a $(q^2-1)_2$-root of unity in $\F_{q^2}^\times$, by \prettyref{lem:beta2}, since the $y$ found there has multiplicative order $(p-1)_2=2$.  Then taking $\delta_k\in\F_{q^2}^\times$ to be the primitive $(q+1)_2$ root of unity $\delta_k=(\beta_k)_2^2$, we may use 
 \prettyref{lem:SFV3.2} to replace $x_k$ with an element whose determinant is $\pm(\beta_k)_2^k\delta_k^{rk}=\pm(\beta_k)_2^{k(2r+1)}$ for any odd $r$, yielding that we may replace each $x_k$ with an element whose determinant is $\pm\beta_2^k$ for a fixed $(q^2-1)_2$-root of unity $\beta_2$.  Hence the resulting $x$ satisfies $\det(x)=\pm\beta_2^n$, which has the stated order.
\end{proof}
 
\section{Application to GGGRs} \label{sec:MoreGGGR}
 
 Here we keep the notation of \prettyref{sec:GGGR} and return to the more general case that $\wt{G}=GL_n^\epsilon(q)$ and $G=SL_n^\epsilon(q)$ for $\epsilon\in\{\pm1\}$.  
 
 \begin{lemma}\label{lem:varphiuconj}
 Let $u\in \mathcal{C}\cap U_{d,2}$ and suppose that $x$ is an element normalizing $U_{d,2}$ and conjugating $u$ to $u^b$.  Then $\varphi_u^x=\varphi_u^b$.
 \end{lemma}
 \begin{proof}
This follows from the construction of $\varphi_u$ in \cite[Section 5]{Taylor16} or \cite[Section 2]{Geck04}.  Indeed, for each $g$ in $U_{d,2}$, we have $\varphi_u^x(g)=\varphi_u(xgx^{-1})=\varphi_{x^{-1}ux}(g)=\varphi_{u^x}(g)=\varphi_{u^b}(g)=\varphi_u(g)^b$, where the second equality is noted in \cite[Remark 2.2]{Geck04}.  
\end{proof}
 
 \begin{lemma}\label{lem:conjinP}
 Let $u\in \mathcal{C}\cap U_{d,2}$ and $\epsilon=-1$.  Then the elements $x$ found in \prettyref{cor:unipconj} and \prettyref{lem:unipconj0mod4} are members of $P_{d}$, and hence normalize $U_{d,2}$.
 \end{lemma}
 \begin{proof}
 First, note that $C_{\wt{\bg{G}}}(u)\leq \bg{P}_{d}$.  Indeed, this is noted in  \cite[Theorem 2.1.1]{Kawanaka86} for simply connected groups, and here we have $\wt{\bg{G}}=\bg{G}Z(\wt{\bg{G}})$.  Further, $u$ is conjugate to $u^b$ in $\bg{P}_{d}$ by \cite[Lemma 4.6]{SFTaylorTypeA}.  So $u^x=u^b=u^y$ for some $y\in \bg{P}_{d}$, which yields that $xy^{-1}\in C_{\wt{\bg{G}}}(u)$, and hence $x\in \wt{G}\cap \bg{P}_{d}$.  This shows that $x$ is contained in $P_{d}$, which contains $U_{d,2}$ as a normal subgroup.
 \end{proof}
 
 \begin{lemma}\label{lem:GGGRarg}
 Let $G=SL^\epsilon_n(q)$ and  $\wt{G}=GL^\epsilon_n(q)$, with $\epsilon\in\{\pm1\}$. Let $\wt{\chi}:=\wt{\chi}_\lambda\in\irr(\wt{G})$ and let $\wt{\Gamma}_u=[U_{1,d}:U_{2,d}]^{-1/2}\ind_{U_{d,2}}^{\wt{G}}(\varphi_u)$ be a Generalized Gelfand-Graev character of $\wt{G}$ such that $\langle \wt{\Gamma}_u, \wt{\chi}\rangle_{\wt{G}}=1$.  Let $\sigma\in\mathrm{Gal}(\bar{\Q}/\F_\lambda)$ and let $x\in\wt{G}$ normalizing $U_{d,2}$ such that $\sigma\varphi_u=\varphi_u^x$.  Then for $\chi\in\irr(G|\wt{\chi})$, there is some conjugate $\chi_0$ of $\chi$ such that $\sigma\chi_0=\chi_0^x$.
 \end{lemma}
 \begin{proof}
 Let $\Gamma_u$ be such that $\wt{\Gamma}_u=\ind_G^{\wt{G}}\Gamma_u$ and  $\Gamma_u=r\cdot\ind_{U_{d,2}}^{{G}}(\varphi_u)$ where $r=[U_{1,d}:U_{2,d}]^{-1/2}$.  
 Then by Clifford theory and Frobenius reciprocity, there is a unique conjugate, $\chi_0$, of $\chi$ such that $\chi_0\in\irr(G|\wt{\chi})$ and $\langle \Gamma_u, \chi_0\rangle_G=1$.  
 Since $\res_G^{\wt{G}}(\wt{\chi})$ is fixed by $\sigma$, we also see $\sigma\chi_0$ is the unique member of $\irr({G}|\wt{\chi})$ satisfying $\langle \sigma\Gamma_u, \sigma\chi_0\rangle_G=1.$  But note that \[\sigma\Gamma_u=r\cdot\ind_{U_{d,2}}^{G}(\sigma\varphi_u)=r\cdot\ind_{U_{d,2}}^{G}(\varphi_u^x)=\Gamma_u^x.\]
  Then
 $\langle \sigma\Gamma_u, \chi_0^x\rangle_{G}=\langle \Gamma_u^x, \chi_0^x\rangle_G=1,$  forcing $\chi_0^x=\sigma\chi_0$ by uniqueness, since $\chi_0^x\in\irr(G|\wt{\chi})$.
 \end{proof}
 
 \section{Main Results} \label{sec:Main}
  
We begin by stating our main results.  The first is an extension of \cite[Theorem 4.8]{turull01} to the case of $SU_n(q)$, describing the field of values $\Q(\chi)$ for each $\chi\in\irr(G)$.

 \begin{theorem}\label{thm:turullext1}
 Let $G=SL^\epsilon_n(q)$ and  $\wt{G}=GL^\epsilon_n(q)$, with $\epsilon\in\{\pm1\}$.  Let $\lambda\in \mathcal{F}_n$ and let $\chi\in\irr(G|\wt{\chi}_\lambda)$.  Then $\Q(\chi)=\F_\lambda$ unless all of the following hold:
\begin{itemize}
\item  $p$ is odd,
\item $q$ is not square,
\item $2\leq n_2\leq (q-\epsilon)_2$, and
\item $\alpha\lambda=\lambda$ for any element $\alpha\in\widehat{T}_1$ of order $n_2$.
\end{itemize} In the latter case,  $\Q(\chi)=\F_\lambda(\sqrt{\eta p})$, where $\eta\in\{\pm1\}$ and $p\equiv\eta\pmod4$.  
 \end{theorem}
 
 Taking into consideration \prettyref{rem:sigmainv}, \prettyref{thm:turullext1} immediately yields the following extension of \cite[Proposition 6.2]{turull01}.

 \begin{corollary}\label{cor:turullext}
 Let $G=SL^\epsilon_n(q)$ and  $\wt{G}=GL^\epsilon_n(q)$, with $\epsilon\in\{\pm1\}$.  Let $\lambda\in \mathcal{F}_n$ and let $\chi\in\irr(G|\wt{\chi}_\lambda)$.   Then the following are equivalent:
 
 \begin{itemize}
 \item $\chi$ is real-valued.
 \item There exists some $\alpha' \in\widehat{T}_1$ 
 such that $\sigma_{-1}\lambda=\alpha'\lambda$, and if $p$ is odd, $q$ is not a square, $2\leq n_2\leq (q-\epsilon)_2$, and $\alpha\lambda=\lambda$ for any element $\alpha\in\widehat{T}_1$ 
 of order $n_2$, then $p\equiv1\pmod4$.
 \end{itemize}
 \end{corollary}

The remainder of this section will be devoted to proving \prettyref{thm:turullext1}.  We begin with an observation restricting the situation of \prettyref{cor:unipconj}. 
 
 \begin{proposition}\label{prop:oddelemoddindex}
Let $G=SL^\epsilon_n(q)$ and  $\wt{G}=GL^\epsilon_n(q)$, with $\epsilon\in\{\pm1\}$. Let $\wt{\chi}\in\irr(\wt{G})$ and let $\wt{\Gamma}_u$ be a GGGR of $\wt{G}$ such that $\langle \wt{\Gamma}_u, \wt{\chi}\rangle_{\wt{G}}=1$. Further, assume that $u$ has an elementary divisor of the form $(t-1)^k$ with $k$ odd.  Then $[\wt{G}:I]$ is odd, where $I=\stab_{\wt{G}}(\chi)$ for any $\chi\in\irr({G}|\wt{\chi})$.
  In particular, in this case, $\F_\lambda=\Q(\chi)$ by \prettyref{lem:sgn}.
 \end{proposition}
 \begin{proof}
 Write $\wt{\Gamma}_u=[U_{1,d}:U_{2,d}]^{-1/2}\ind_{U_{d,2}}^{\wt{G}}(\varphi_u)$.  By \prettyref{lem:SFV3.2}, there is some $x\in C_{\wt{G}}(u)$ with determinant $\delta^k$, where $\delta$ is a $(q+1)_2$-root of unity in $\F_{q^2}^\times$.  In particular, $|\det(x)|=(q+1)_2$ since $k$ is odd, and $\varphi_u^x=\varphi_u$ since $x$ normalizes $U_{d,2}$ as in the proof of \prettyref{lem:varphiuconj}.  Then applying \prettyref{lem:GGGRarg} with $\sigma$ trivial yields that some conjugate $\chi_0$ of $\chi$ satisfies $\chi_0^x=\chi_0$.  This implies $[I:G]$ is divisible by $(q+1)_2$, so that $[\wt{G}:I]$ must be odd.
 \end{proof}
 
For the remainder of this section, we will consider the case  $\epsilon=-1$, so that $\wt{G}=GU_n(q)$ and $G=SU_n(q)$.  In particular, \prettyref{prop:oddelemoddindex} yields that if $[\wt{G}:I]$ is even, then neither condition in \prettyref{cor:unipconj} holds if $n$ is divisible by $4$, and condition 1 holds if $n\equiv 2\pmod 4$, taking into account \prettyref{rem:condsnmod4}. 
 
 \begin{proposition}\label{prop:n2mod4}
 Let $\epsilon=-1$ and suppose that $q\equiv \eta\mod 4$ is nonsquare with $\eta\in\{\pm1\}$ 
 and that $n\equiv 2\pmod 4$.  Then the converse of \prettyref{lem:sgn} holds.  That is, for $\chi\in\irr(G|\wt{\chi}_\lambda)$, $\F_\lambda=\Q(\chi)$ if and only if $\mathrm{sgn} \lambda\neq\lambda$.  Alternatively, $\F_\lambda(\sqrt{\eta p})=\Q(\chi)$ if and only if $\mathrm{sgn} \lambda = \lambda$.
 \end{proposition}
 \begin{proof}

 We must show that if $\mathrm{sgn}\lambda= \lambda$, then $\F_\lambda\neq \Q(\chi)$.  First, recall that this condition on $\lambda$ is equivalent to the condition that $[\wt{G}:I]=[T_1:\mathcal{I}(\lambda)]$ is even, by \prettyref{lem:orbiteven}. Since $n_2=2$,  \prettyref{prop:indstabdivides} yields that $[\wt{G}:I]_2=2$.    
 This means that no $\wt{G}$-conjugate of $\chi$ can be fixed by any  $\wt{g}\in \wt{G}$ whose determinant satisfies $|\det(\wt{g})|_2=(q+1)_2$.
 
As an abuse of notation, we let $\tau$ also denote the unique element of $\mathrm{Gal}(\F_\lambda(\zeta_p)/\F_\lambda)$ that restricts to our fixed generator $\tau$ of $\mathrm{Gal}(\Q(\zeta_p)/\Q)$.
In the notation of \prettyref{lem:GGGRarg}, we have $\varphi_u^x=\varphi_u^b=\tau\varphi_u$ for some $x\in P_{d}$ satisfying $|\det(x)|=(q+1)_2$, by \prettyref{cor:unipconj} and Lemmas \ref{lem:varphiuconj} and \ref{lem:conjinP}.  Then by \prettyref{lem:GGGRarg}, there is a conjugate $\chi_0$ of $\chi$ such that $\chi_0^x=\tau\chi_0$.  In particular, note that the condition on the determinant yields that $\chi_0^x\neq \chi_0$, so $\tau\chi_0\neq\chi_0$.  Since $\chi$ and $\chi_0$ have the same field of values, we see $\tau\chi\neq \chi$, and we have $\F_\lambda\neq \Q(\chi)$.
  \end{proof}
 
 \begin{proposition}\label{prop:n0mod4}
 Let $\epsilon=-1$ and suppose that $q\equiv 3\pmod 4$ and $4\leq n_2\leq (q+1)_2$, and let $\chi\in\irr(G|\wt{\chi}_\lambda)$ and $I:=\stab_{\wt{G}}(\chi)$. Then $\F_\lambda=\Q(\chi)$ if and only if $[\wt{G}:I]_2<n_2$.
 \end{proposition}
 \begin{proof}
 First, note that $[\wt{G}:I]_2\leq n_2$ by \prettyref{prop:indstabdivides}.  Note that by \prettyref{lem:sgn}, we may assume that $[\wt{G}:I]$ is even and therefore that $\chi\in\irr(G|\wt{\Gamma}_u)$ where $u$ has no odd-power elementary divisor, by \prettyref{prop:oddelemoddindex}.
 By Lemmas \ref{lem:unipconj0mod4}, \ref{lem:varphiuconj}, and \ref{lem:conjinP}, there is some $x\in \wt{G}$ such that $\varphi_u^x=\varphi_u^b=\tau\varphi_u$ and $|\det(x)|=\frac{2(q+1)_2}{n_2}$, which is divisible by $2$ since $n_2\leq(q+1)_2$.  By \prettyref{lem:GGGRarg}, there is a conjugate $\chi_0$ of $\chi$ such that $\chi_0^x=\tau\chi_0$.
  
  Suppose first that $[\wt{G}:I]_2=n_2$, so that $x$ cannot stabilize $\chi_0$, since $[I:G]_2=\frac{(q+1)_2}{n_2}$ (and the same is true for the stabilizer of $\chi_0$).   This yields that $\chi_0\neq \tau\chi_0$, so the same holds for $\chi$.  Hence if $[\wt{G}:I]_2=n_2$, then $\F_\lambda\neq\Q(\chi)$.
 
 Now suppose $[\wt{G}:I]_2<n_2$.  That is, $[\wt{G}:I]_2\leq \frac{n_2}{2}$.  Then the stabilizer of $\chi_0$ must contain $x$, since $I/G\cong \mathcal{I}(\lambda)$ is cyclic and contains the unique subgroup of $\wt{G}/G$ of size $\frac{2(q+1)_2}{n_2}$.  Then $\chi_0=\tau\chi_0$, and the same is true for $\chi$, so $\F_\lambda=\Q(\chi)$.
 \end{proof}

 \begin{proof}[Proof of \prettyref{thm:turullext1}]
Note that the case $[\wt{G}:I]_2=n_2$, for any $n$ even, is equivalent to having $\alpha\lambda=\lambda$ for any $\alpha\in \wh{T}_1$ of order $n_2$.  In case $n\equiv 2\pmod 4$, we remark that this $\alpha$ is $\mathrm{sgn}$.  Hence Propositions \ref{prop:initialcase}, \ref{prop:n2mod4}, and  \ref{prop:n0mod4} combine to yield the statement.
 \end{proof}

\section*{Acknowledgements}

The authors were each supported in part by grants from the Simons Foundation (Awards \#351233 and \#280496, respectively).

Part of this work was completed while the first-named author was in residence at the Mathematical Sciences Research Institute in Berkeley, California during the Spring 2018 semester program on Group Representation Theory and Applications, supported by the National Science Foundation under Grant No. DMS-1440140.  She thanks the institute and the organizers of the program for making her stay possible and providing a collaborative and productive work environment.

\bibliographystyle{alpha}
\bibliography{researchreferences}
\end{document}